\documentclass[12pt]{amsart} 
\usepackage{geometry}             

\usepackage{float}
\usepackage{fullpage}
\usepackage[mathscr]{euscript}
\usepackage[all]{xy}
\usepackage{epsfig}
\usepackage[T1]{fontenc}
\usepackage{amsfonts}
\usepackage{caption}

\usepackage{graphicx}
\usepackage{amssymb}
\usepackage{amsmath}
\usepackage{amsthm}
\usepackage{mathrsfs}
\usepackage{epstopdf}
\usepackage{url}

\usepackage[dvipsnames,usenames, table]{xcolor}
\usepackage[colorlinks=true,urlcolor=blue,linkcolor=blue,citecolor=blue]{hyperref}
\hypersetup{
	linkbordercolor={1 0 0}, 
	citebordercolor={0 1 0} 
}
\usepackage{cite}
\usepackage[noabbrev]{cleveref}

\usepackage{color}
\usepackage{soul} 
\setstcolor{red}

\usepackage{tikz}

\newcommand{\ga}{\gamma}

\newcommand{\la}{\lambda}

\newcommand{\De}{\Delta}
\newcommand{\Si}{\Sigma}
\newcommand{\ZZ}{{\mathbb Z}}

\newcommand{\cD}{\mathcal D}
\newcommand{\cE}{\mathscr E}

\newcommand{\cL}{\mathcal L}

\newcommand{\Tr}{\operatorname{Tr}}

\newcommand{\vlk}{\operatorname{{\it v}\ell{\it k}}}
\newcommand{\Int}{\operatorname{Int}}
\newcommand{\sm}{\smallsetminus}
\newcommand{\co}{\colon}
\newcommand{\lb}{\left\langle}
\newcommand{\rb}{\right\rangle}
\newcommand{\Tube}{\operatorname{Tube}}

\newtheorem{theorem}{Theorem}[section]

\newtheorem{proposition}[theorem]{Proposition}

\newtheorem{corollary}[theorem]{Corollary}

\newtheorem*{theorem*}{Theorem}

\theoremstyle{definition}     
\newtheorem{definition}[theorem]{Definition}

\theoremstyle{remark}
\newtheorem{remark}[theorem]{Remark}
\newtheorem{example}[theorem]{Example}
\newtheorem{problem}[theorem]{Problem}

\title[Classical results for alternating virtual links]{Classical results for alternating virtual links}
\author[Hans U. Boden]{Hans U. Boden}
\address{Mathematics \& Statistics, McMaster University, Hamilton, Ontario}
\email{boden@mcmaster.ca}

\author[Homayun Karimi]{Homayun Karimi}
\address{Mathematics \& Statistics, McMaster University, Hamilton, Ontario}
\email{karimih@mcmaster.ca}


\begin{document}

\begin{abstract}

We extend some classical results of Bankwitz, Crowell, and Murasugi to the setting of virtual links. For instance, we show that an alternating virtual link is split if and only if it is visibly split, and that the Alexander polynomial of any almost classical alternating virtual link is alternating. The first result is a consequence of an inequality relating the link determinant and crossing number for any non-split alternating virtual link.  The second is a consequence of the matrix-tree theorem of Bott and Mayberry.
We extend the first result to semi-alternating virtual links. We discuss the Tait conjectures for virtual and welded links and note that Tait's second conjecture is not true for alternating welded links.

\end{abstract}

\maketitle
\setcounter{section}{1} \noindent
\subsection{Introduction}  \label{sec-1} 

In this paper, we establish conditions that are satisfied by invariants of alternating virtual links, such as the link determinant and Alexander polynomial. As an application, we deduce that a reduced alternating virtual link diagram is split if and only if it is visibly split.

A  link is said to be \textit{alternating} if it admits an alternating diagram, and a diagram is alternating if the crossings alternate between over and under-crossing as one travels around any component. This applies to classical and virtual links, with the proviso that virtual crossings are ignored.
 
In \cite{Bankwitz}, Bankwitz proved that $\det(L) \geq c(L)$ for any non-split alternating link $L$, where $\det(L)$ denotes the link determinant and $c(L)$ the crossing number of $L$. In \cite{Crowell, Murasugi-1958}, Crowell and Murasugi independently proved that the Alexander polynomial of an alternating link is alternating. Here, a Laurent polynomial $\Delta_L(t)=\sum c_i t^i$ is said to be \textit{alternating} if its coefficients satisfy $(-1)^{i+j} c_i c_j \geq 0$. 

We extend the results of Bankwitz, Crowell, and Murasugi to alternating virtual links. Virtual knots were introduced by Kauffman in \cite{KVKT}, and they represent a natural generalization to knots in thickened surfaces up to stabilization. Classical knots embed faithfully into virtual knot theory \cite{GPV}, and many  invariants from classical knot theory extend in a natural way to the virtual setting.

For example, the link determinant $\det(L)$ is defined in terms of the coloring matrix and extends to checkerboard colorable virtual links (defined below). The link $L$ admits a $p$-coloring if and only if $p$ divides $\det(L)$. 
One of our main results is that $\det(L) \geq c(L)$ for any non-split alternating virtual link $L,$ where $c(L)$ is the \textit{classical crossing number} of $L$. This result applies to show that a reduced alternating virtual link diagram is split if and only if it is visibly split.

The Alexander polynomial $\Delta_L(t)$  is defined in terms of the Alexander module of $L$, and it extends to almost classical virtual links (defined below). Another one of our main results is that, for any reduced alternating link $L$ that is almost classical, its Alexander polynomial $\Delta_{L}(t)$ is alternating. To prove this result, we appeal to the Matrix-Tree Theorem. It applies to show that many virtual knots cannot be represented by alternating diagrams.

The link determinant and Alexander polynomial are both invariant under welded equivalence. Therefore, our main results can be seen as providing restrictions on a virtual link diagram for it to be welded equivalent to an alternating virtual link. This is discussed at the end of the paper, where we state open problems related to the Tait conjectures for welded links.

We provide a brief synopsis of the contents of the rest of this paper.
In \ref{sec-2}, we review background material on links in thickened surfaces and virtual and welded links, together with Cheng coloring and Alexander numbering for virtual links. 
In \ref{sec-3}, we review the link group and determinant. 
In  \ref{sec-4}, we recall the Matrix-Tree Theorem, which is used to prove one of the main results.
In \ref{sec-5}, we prove that split alternating virtual links are visibly split. 
In \ref{sec-6}, we prove analogous results for semi-alternating links,
and in \ref{sec-7}, we present a discussion on the Tait conjectures for welded links and state some interesting open problems. 


\setcounter{theorem}{0}  
\setcounter{section}{2} \noindent
\subsection{Virtual links} \label{sec-2}

In this section we review the basic properties of virtual  links,
including Gauss diagrams, links in thickened surfaces, welded links, ribbon torus links, alternating virtual links, virtual linking numbers, Cheng colorings, and Alexander numberings.

\medskip\noindent{\bf Virtual link diagrams.} Virtual links are defined as equivalence classes of virtual link diagrams. Here, a virtual link diagram is an  immersion of one or more circles in the plane with only finitely many regular singularities, each of which is a double point.  Each double point is either classical (indicated by over- and under-crossings) or virtual (indicated by a circle). Two diagrams are said to be virtually equivalent if they can be related by planar isotopies and a series of \textit{generalized Reidemeister moves} ($r1$)--($r3$) and ($v1$)--($v4$) depicted in \Cref{VRM}. 

An orientation for a virtual link is obtained by choosing an orientation for each component. For a diagram $D$, the orientation is usually indicated by placing one arrow on each component of $D$.

\begin{figure}[ht]
\centering
\def\svgwidth{250pt}
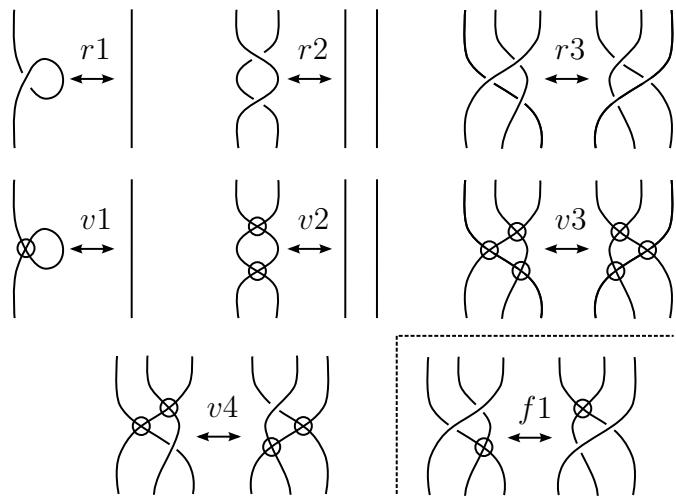
\caption{\small The generalized Reidemeister moves ($r1$)--($r3$), ($v1$)--($v4$) and the forbidden move ($f1$).}
\label{VRM}
\end{figure}

\medskip\noindent{\bf Gauss diagrams.} Virtual links can also be defined as equivalence classes of Gauss diagrams, which consist of one or more circles traversed counterclockwise, together with directed chords on the circles representing the classical crossings. The chords point from over-crossings to under-crossings and are decorated with a sign ($+$ or $-$) to indicate whether the crossing is positive or negative. Each virtual link diagram determines a Gauss diagram, and vice versa, and this correspondence is well-defined up to moves ($v1$)--($v4$). The Reidemeister moves can be translated into moves between Gauss diagrams, and in this way a virtual link can be regarded as an equivalence class of Gauss diagrams. By convention, the core circles of a Gauss diagram are oriented counterclockwise. 

Notice that the Gauss diagram does not keep track of the virtual crossings. In effect, the virtual crossings are not really there, rather they are an inevitable consequence of trying to represent a non-planar virtual link diagram by a diagram in the plane.

A virtual link diagram is said to be \textit{split} if its associated Gauss diagram is disconnected, and
a virtual link is \textit{split} if it can be represented by a split diagram.
For classical links, this agrees with the usual definition. For virtual links, a diagram can be split and connected. However, any diagram that is split can be transformed into a disconnected diagram using moves ($v1$)--($v4$).

\medskip\noindent{\bf Links in thickened surfaces.}
A third approach is to define virtual links as stable equivalence classes of links in thickened surfaces, and we take a moment to explain this. 

Let $\Si$ be a closed, oriented surface and $I =[0,1]$.
Consider a link $\cL \subset \Si \times I$ in the thickened surface, up to isotopy. 
Let $p\co \Si \times I \to \Si$ be the projection map.

Stabilization is the operation of adding a 1-handle to $\Si$, disjoint from $p(\cL)$, and destabilization is the opposite procedure.
Two links $\cL \subset \Si\times I$ and $\cL' \subset \Si' \times I$ are said to be \textit{stably equivalent} if one is obtained from the other by a finite sequence of stabilizations, destablizations, and orientation-preserving diffeomorphisms of the pairs 
$(\Si\times I, \Si \times \{0\})$ and $(\Si'\times I, \Si' \times \{0\})$.
In \cite{Carter-Kamada-Saito}, Carter, Kamada, and Saito show there is a one-to-one correspondence between virtual links and stable equivalence classes of links in thickened surfaces.

Thus, every virtual link can be represented  as a link in a thickened surface. Further, any such link itself can be represented as a link diagram on $\Si.$ A \textit{link diagram} $\cD$ on $\Si$ is a tetravalent graph with over= and under-crossing information drawn at each vertex in the usual way.

A link diagram $\cD$ on $\Si$ is said to be a \textit{split diagram} if it is disconnected, and
a link in $\Si \times I$ is said to be \textit{split} if it can be represented by a split diagram.

\medskip\noindent{\bf Welded links.} Two virtual links are said to be \textit{welded equivalent} if one can be obtained from the other by a sequence of generalized Reidemeister moves and the first forbidden move ($f1$) as depicted in \Cref{VRM}. 
In terms of Gauss diagrams, the first forbidden move corresponds to exchanging two adjacent arrow feet without changing their signs or arrowheads, see \Cref{f1-r1}.  Therefore, a welded link can also be viewed as an equivalence class of Gauss diagrams. 
 
\begin{figure}[ht]
\centering
\includegraphics[scale=0.80]{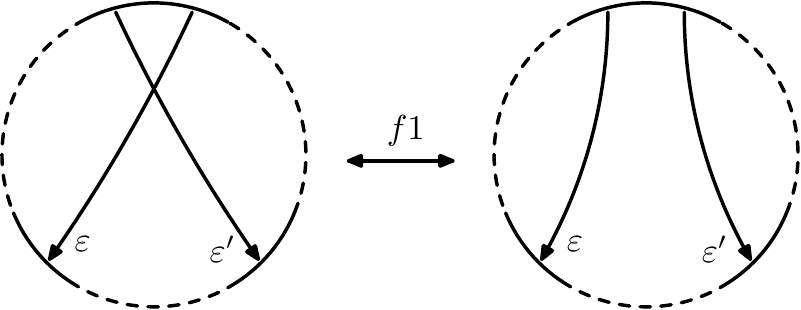}
\caption{\small The forbidden overpass $(f1)$ for Gauss diagrams.}
\label{f1-r1}
\end{figure}

\medskip\noindent{\bf Ribbon torus links.}
Every welded link determines a ribbon knotted surface in $S^4$. This is based on a beautiful construction by Satoh \cite{Satoh}, which associates to a welded link $L$ a ribbon torus link $\Tube(L)$ in $S^4$. In \cite{Satoh}, Satoh shows that every ribbon torus link occurs as $\Tube(L)$ for some welded link, and  that $\pi_1(S^4  \sm \Tube(L))$  is isomorphic to the link group $G_L$, defined below.

The correspondence between welded links and ribbon torus links is not one-to-one,  see \cite{Winter}. It is an open problem to determine necessary and sufficient conditions for two welded knots to represent the same ribbon torus knot (cf. \cite[Question 3.6]{Audoux}).

\medskip\noindent{\bf Alternating virtual links.} 
A virtual link diagram $D$ is called \textit{alternating} if the classical crossings alternate between over-crossing and under-crossing as we go around each component. A Gauss diagram is alternating if it alternates between arrow heads and tails when going around each of the core circles. A virtual or welded link is called \textit{alternating} if it can be represented by an alternating virtual link diagram. For example, consider the virtual links in \Cref{vhopf}. The virtual Borromean rings is alternating, but the virtual Hopf link is not.

\medskip\noindent{\bf Virtual linking numbers.} 
If $J$ and $K$ are oriented virtual knots, then the virtual linking number $\vlk(J, K)$ is defined as the sum of the writhe of the classical crossings where $J$ goes over $K$. Using the same definition, we can define $\vlk(J, K)$ more generally
when $J$ and $K$ are oriented virtual links. Note that $\vlk(J, K)$ is additive, namely if $J = J' \cup J''$, then $\vlk(J, K) = \vlk(J', K) + \vlk(J'', K)$, and likewise if $K = K' \cup K''.$
The virtual linking numbers are not symmetric, i.e., it is not generally true that $\vlk(J, K) = \vlk(K, J)$. For example, consider the oriented virtual links in \Cref{vhopf}. For the virtual Hopf link, we have $\vlk(J,K)=1$ and $\vlk(K,J)=0$, and for the virtual Borromean rings, we have $$\vlk(I,J)=\vlk(J,K)=\vlk(K,I)=0,\; \vlk(J,I)=\vlk(K,J)=1, \text{ and } \vlk(I,K)=-1.$$  

\begin{figure}[ht]
\includegraphics[scale=1.00]{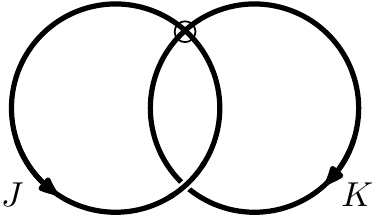} \qquad \qquad  
\includegraphics[scale=1.00]{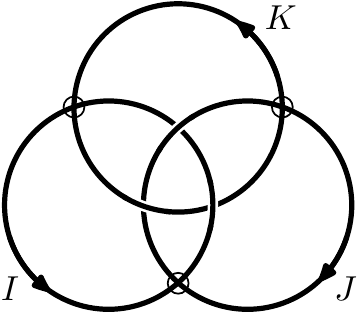}  
\caption{\small The virtual Hopf link and the virtual Borromean rings.} \label{vhopf}
\end{figure}

\medskip\noindent{\bf Cheng colorings.} 
Given a virtual link diagram $D$, a \textit{Cheng coloring} of $D$ is an assignment of integer labels to each arc of $D$ that satisfies the local rules in \Cref{fig-Cheng}. An elementary exercise shows if $D$ and $D'$ are two virtual link diagrams that are related by virtual Reidemeister moves, then $D$ admits a Cheng coloring if and only if $D'$ does. A virtual link $L$ is said to be \textit{Cheng colorable} if it can be represented by a virtual link diagram with a Cheng coloring. 

Not all virtual links are Cheng colorable. For example, the virtual Hopf link in \Cref{vhopf} is not Cheng colorable. More generally, given a virtual link $L = K_1 \cup \cdots \cup K_m$ with $m$ components, then an elementary argument shows that $L$ admits a Cheng coloring if and only if it satisfies $$\vlk(K_i,L\sm K_i)=\vlk(L\sm K_i, K_i)=0$$ for each $i=1,\ldots, m.$

\begin{figure}[h]
\centering
\includegraphics[scale=0.90]{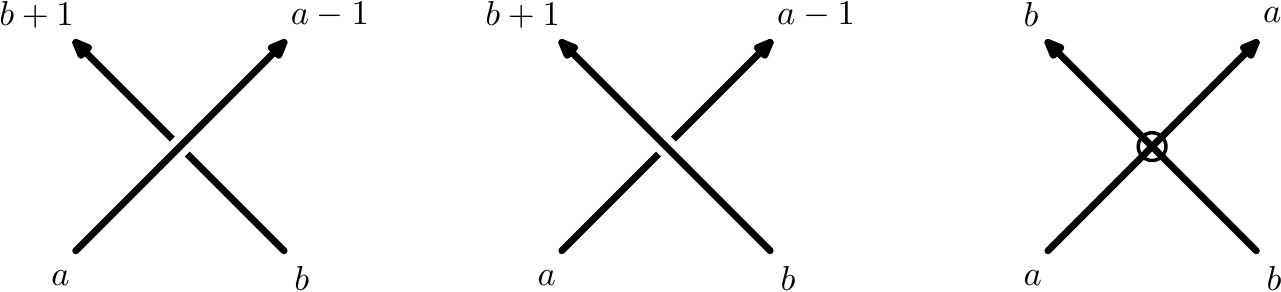}
\caption{\small Local rules for a Cheng coloring at classical and virtual crossings.}
\label{fig-Cheng}
\end{figure}

\medskip\noindent{\bf Alexander numberings and almost classical links.} 
Given a virtual link diagram $D$, an \textit{Alexander numbering} on $D$ is an assignment of integer labels to each arc of $D$ that satisfies the local rules in \Cref{fig-Alexander}. If $D$ admits an integer labeling that satisfies the local rules mod $p$, then $D$ is said to be mod $p$ Alexander numberable. 

Notice that if $D$ is Alexander numberable, then it is Cheng colorable. Conversely, a virtual link diagram $D$ is Alexander numberable if and only if it admits a Cheng coloring such that the arc labels satisfy $b=a-1$ at each classical crossing.  Likewise, a virtual link diagram $D$ is mod $p$ Alexander numberable if and only if it admits a Cheng coloring such that the arc labels satisfy $b\equiv a-1$ (mod $p$) at each classical crossing.

\begin{figure}[h]
\centering
\includegraphics[scale=0.90]{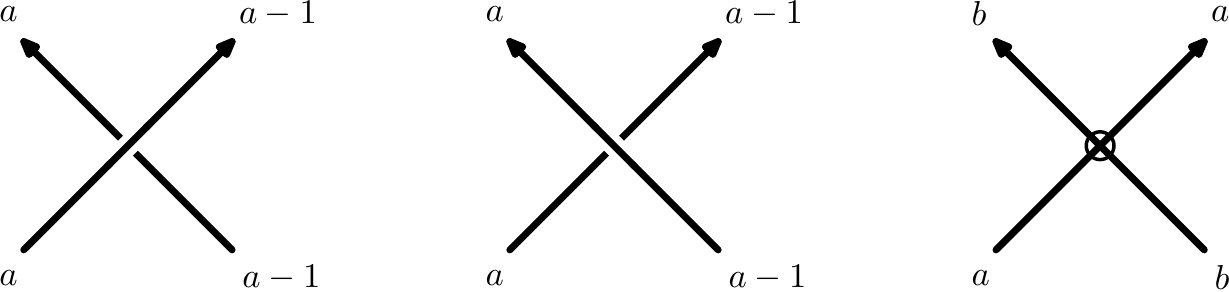}
\caption{\small Local rules for an Alexander numbering  at classical and virtual crossings.}
\label{fig-Alexander}
\end{figure}

A virtual link is said to be \textit{almost classical} if it admits an Alexander numberable diagram, and it is said to be \textit{checkerboard colorable} if it admits a mod 2 Alexander numberable diagram.

Recall from \cite[Theorem 6.1]{Boden-Gaudreau-Harper-2017} that a virtual link is almost classical if and only if it can be represented by a null-homologous link $\cL \subset \Si \times I$, or equivalently if $\cL$ admits a Seifert surface. 

Recall also from \cite[Proposition 1.1]{Boden-Chrisman-Karimi-2021} that a virtual link is checkerboard colorable if and only if it can be represented by a $\ZZ/2$ null-homologous link $\cL \subset \Si \times I$, or equivalently if $\cL$ admits an unoriented spanning surface. This is the case if and only if $\cL$ can be represented by a checkerboard colorable diagram on $\Si$. 

If a virtual link admits a diagram which is Cheng colorable, then any diagram for the same link is also Cheng colorable. The reason is that Cheng colorings extend along generalized Reidemeister moves. The same thing is not true for Alexander numberings. Indeed, one can easily find two virtual link diagrams for the same link such that one of them is Alexander numberable and the other is not. Thus, Alexander numberings of virtual links do not always extend along generalized Reidemeister moves.

However, if two virtual knot diagrams are Alexander numberable and are related by generalized Reidemeister moves, then one can arrange that they are related through Alexander numberable diagrams.  More precisely, suppose $D$ and $D'$ are two virtual knot diagrams and
\begin{equation} \label{eq-chain}
D =D_1 \sim D_2 \sim \cdots \sim D_r = D'\
\end{equation}
is a chain of diagrams, where  $D_{i+1}$ is obtained from $D_i$ by a single generalized Reidemeister move. If $D$ and $D'$ are Alexander numberable, then there is a chain \eqref{eq-chain} such that each $D_i$ is Alexander numberable. A similar result holds if $D$ and $D'$ are assumed to be mod $p$ Alexander numberable. These statements can be proved using parity projection, see \cite{Manturov, Nikonov-2016}. Any minimal crossing diagram of an almost classical link is Alexander numberable, and this can also be proved using parity projection, see \cite{Rushworth-2021, Boden-Rushworth}.

The corresponding statements for welded knots and links are either not true, or not known to be true. 

\setcounter{section}{3} \noindent
\subsection{Link group, Alexander module, and determinant} \label{sec-3}
In this section, we introduce the link group and the Alexander module associated to a virtual link. We also recall the definition of the link determinant associated to a checkerboard colorable virtual link and show that $\det(L)=0$ when $L$ is split. Finally, we discuss mod $p$ labelings of virtual knots and show that $K$ admits a mod $p$ labeling if and only if $\det(K)=0$ (mod $p$).

\medskip\noindent{\bf Link Group.}  
For classical links, the link group is just the fundamental group of the complement of the link. For a link $L$, this group is denoted $G_L$. Thus, $G_L =\pi_1(X_L)$ where $X_L$ is the result of removing an open tubular neighborhood of $L$ from $S^3$. 

As an invariant of classical knots, the knot group is an unknot detector, indeed the only classical knot $K$ whose knot group  is infinite cyclic is the trivial knot. In fact, Waldhausen's theorem implies  that the knot group together with its peripheral structure is a \textit{complete invariant} of classical knots, which is to say that two classical knots are equivalent if and only if they have isomorphic knot groups with equivalent peripheral structures.

The link group generalizes in a natural way to give an invariant of virtual links by means of Wirtinger presentations. In fact, the abstract group together with its peripheral structure are invariant under the first forbidden move and thus define invariants of the underlying welded link.

Given a virtual link diagram for $L$, we will describe the \textit{Wirtinger presentation} of $G_L$.  Let $D$ be a regular projection of $L$, and suppose it has $n$ classical crossings. We define a \textit{long arc} of the diagram $D$ to be one that goes from one classical under-crossing to the next, passing through virtual crossings as it goes. Enumerate the long arcs of $D$ by $x_{1},\ldots,x_m$ and the classical crossings by $c_{1},\ldots,c_n$.

\begin{figure}[h!]
\centering\includegraphics[scale=0.90]{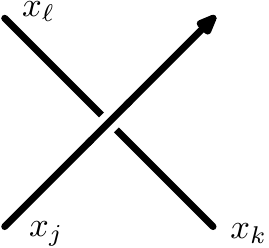}
\caption{\small Arc labels at a crossing.} \label{crossing-group}
\end{figure}

For each crossing, labeled as in \Cref{crossing-group}, we  have the relation $r_i =  x_{\ell}x_{j}^{-1}x_{k}^{-1}x_{j}$. The Wirtinger presentation for $G_{L}$ is then:
\begin{equation} \label{eq:Wirt}
G_{L}=\langle x_{1},\ldots,x_m \mid r_1,\ldots,r_n\rangle.
\end{equation}

\medskip\noindent{\bf Alexander module.}
In order to define the \textit{Alexander module}, we briefly recall \textit{Fox differentiation}. Let $F_m$ be the free group on $m$ generators, so elements of $F_m$ are words in $x_1,\ldots, x_m$. For $j=1,\ldots, m$, the Fox derivative $\partial/\partial x_j$ is an endomorphism of $\ZZ[F_m]$, the group ring, defined so that $\partial/\partial x_j(1) =0$ and 
$$\frac{\partial }{\partial x_{j}}(x_{i}) =\begin{cases} 1 & \text{if $i=j$,} \\ 0 & \text{otherwise.} \end{cases}$$ 
Further, given words $w,z \in F_m$, the Fox derivative satisfies the Leibnitz rule:
$$\frac{\partial }{\partial x_{j}}(wz)=\frac{\partial}{\partial x_{j}}(w)+w\frac{\partial}{\partial x_{j}}(z).$$ 
These relations completely determine $\partial/\partial x_{j}$ on every word $w \in F_m,$ and it is extended linearly to the group ring $\ZZ[F_m]$.

We use this to describe the construction of the Alexander module associated to a link $L$. Let $G_{L}'=[G_L,G_L]$ and $G_{L}''=[G_{L}',G_{L}']$ be the first and second commutator subgroups, then the Alexander module is the quotient $G_{L}'/G_{L}''$. It is a finitely generated module over $\ZZ[t,t^{-1}]$, the ring of Laurent polynomials, and it is determined by the \textit{Fox Jacobian matrix} $A$ as follows. Here, $A$ is the $n \times m$ matrix with $ij$ entry equal to $\left. \frac{\partial r_i}{\partial x_j}\right|_{x_1,\ldots, x_m=t}$. In particular, the Fox Jacobian is obtained by Fox differentiating the relations $r_i$  with respect to the generators $x_{j}$ and applying the abelianization map $x_{j}\mapsto t$ for $j=1,\ldots,m$. We define the $k$-th \textit{elementary ideal} $\cE_k$ as the ideal of $\ZZ[t,t^{-1}]$ generated by all $(n-k)\times(n-k)$ minors of $A$. 

The matrix $A$ depends on the choice of a presentation for $G_{L}$, but the associated sequence of elementary ideals 
$$\{0\}=\cE_0\subset\cE_{1}\subset\ldots\subset\cE_{n}=\ZZ[t,t^{-1}]$$
does not. 

For any classical link $L$, the first elementary ideal $\cE_{1}$ is a principal ideal, and the Alexander polynomial $\Delta_{L}(t)$ is defined as the generator of $\cE_{1}$. The Alexander polynomial is well-defined up to multiplication by $\pm t^{k}$ for $k\in\ZZ$. It is obtained by taking the determinant of the \textit{Alexander matrix}, which is the $(n-1)\times(n-1)$ matrix obtained by removing a row and column from $A$.

For a virtual link $L$, one can mimic the construction of the Alexander module by regarding the quotient $G_L'/G_L''$ as a module over $\ZZ[t,t^{-1}],$  This can be used to define elementary ideals and the Alexander polynomial for virtual links. However, in contrast to the case of classical links, the first elementary ideal may not be principal. One way to remedy the situation is to replace the elementary ideals $\cE_k$ with the smallest principal ideal containing them. For instance, this would suggest a way to define an Alexander polynomial for a virtual link $L$ to be a generator of the principal ideal containing  $\cE_1$. However, since the link group itself is only an invariant of the associated \textit{welded link},
the invariants one obtains in this way will not be very refined. Indeed, for many virtual knots, the Alexander polynomial is trivial.  

\medskip\noindent
{\bf Link determinant for checkerboard colorable virtual links.}
 We review the definition of the link determinant in terms of the coloring matrix and
show that it extends to an invariant of checkerboard colorable virtual links. 
We also prove that the determinant of the coloring matrix is odd for any checkerboard colorable virtual knot, 
and that a checkerboard colorable virtual knot $K$ admits a mod $p$ labeling if and only if $p$ divides $\det(K)$.

Let $L$ be a virtual link that is represented by a checkerboard colorable diagram $D$ with $n$ classical crossings $\{c_1,\ldots, c_n \}$ and $m$ long arcs $\{a_1,\ldots, a_{m}\}.$  If $D$  has $k$ connected components, then $m = n+k-1$.

Define the $n\times m$ coloring matrix $B(D)$ so that its $ij$ entry is given by
\begin{eqnarray*}
b_{ij}(D) &=&\begin{cases}
2,&  \text{if $a_j$ is the over-crossing arc at $c_i$},\\
-1,& \text{if $a_j$ is one of the under-crossing arcs at $c_i$},\\
0,& \text{otherwise}.
\end{cases}
\end{eqnarray*}
In case $a_j$ is coincidentally the over-crossing arc and one of the under-crossing arcs at $c_i$,  then we set $b_{ij}(D)=1$. In that case, if $a_{k}$ is the other under-crossing arc at $c_i$, then we set $b_{ik}(D)=-1$.

Note that the matrix $B(D)$ is the one obtained by specializing the Fox Jacobian matrix $A(D)$ at $t=-1$.\footnote{This is only true up to sign for any given row.} Here, $A(D)$ is defined in terms of taking Fox derivatives of the Wirtinger presentation of the link group $G_D$ whose generators are given by the arcs $a_1,\ldots, a_m$ and relations are given by classical crossings $c_1,\ldots, c_n$ and applying the abelianization homomorphism $G_L \to \lb t \rb, \ a_i \mapsto t$. For details, see \cite[Section 5]{Boden-Gaudreau-Harper-2017}.

Notice that the entries in each row of $B(D)$ sum to zero, therefore, it has rank at most $n-1$.  The proof of the next result is similar to that of \cite[Proposition 2.6]{Boden-Nicas-White}.

\begin{proposition} \label{prop-alex-det}
Any two  $(n-1) \times (n-1)$ minors of $B(D)$ are equal up to sign. The absolute value of the minor is independent of the choice of checkerboard colorable diagram $D$. It defines an  invariant of checkerboard colorable links $L$ called the \textit{determinant} of $L$ and denoted $\det(L)$.
\end{proposition}
\begin{proof}
As previously noted, the columns of $B(D)$ always sum to zero, and we will use checkerboard colorability to derive a linear relation among the rows. Recall that the diagram $D$ is checkerboard colorable if and only if it admits a mod 2 Alexander numbering. For each crossing $c_i$ of $D$, let $\ga_i=(-1)^{\la_i},$ where $\la_i\in\{0,1\}$ is the Alexander number on the incoming under-crossing at $c_i$. Then we claim that one obtains a linear relation on the rows by multiplying the $i$-th row of $B(D)$ by $\ga_i$. 

To see why this is true, notice that the columns of $B(D)$ correspond to arcs of the diagram, and in any given column, there are nonzero entries  for each crossing the arc is involved in. The arc starts and ends with under-crossings, and the associated column entries are both $-1$. Every time the arc crosses over another arc, there is an associated column entry equal to 2. Since the diagram is mod 2 Alexander numberable, the numbers on the transverse arcs alternate between $0$ and $1$ as one travels along the arc. Consequently, the coefficients $\gamma_i$ alternate in sign as one travels along the arc. Therefore, after multiplying the $i$-th row by $\ga_i$, this shows that the entries in each column sum to zero. Furthermore, since each coefficient $\ga_i$ is a unit, every row of $B(D)$ is a linear combination of the other rows. This shows that the $(n-1) \times (n-1)$ minors of $B(D)$ are all equal up to sign.
\end{proof}

\begin{proposition}\label{split det}
Suppose $L$ is a checkerboard colorable virtual link. If $L$ is split, then $\det(L)=0$. 
\end{proposition}

\begin{proof}
Suppose $D=D_1 \cup D_2$ is a split checkerboard colorable diagram for $L$. In each row of the coloring matrix, the nonzero elements are either $2,-1,-1$ or $1,-1$. It follows the rows add up to zero. We consider a simple closed curve in the plane which separates $D$ into two parts. It follows that the coloring matrix $B=B(D)$ admits a $2 \times 2$ block decomposition of the form
$$B=\begin{bmatrix}
B_1 & 0 \\ 
0 & B_2
\end{bmatrix},$$
where $B_1$ and $B_2$ are the coloring matrices for $D_1$ and $D_2,$ respectively. Since $\det(B_1) =0=\det(B_2)$, it follows that  the matrix obtained by removing a row and column from $B$ also has determinant zero.      
\end{proof}

Next, we define a mod $p$ labeling for a virtual knot diagram.

\begin{definition}
Let $p$ be a prime number. A link diagram can be labeled mod $p$ if each long arc can be labeled with an integer from $0$ to $p-1$ such that 
\begin{itemize}
\item[(i)] at each crossing the relation $2x-y-z=0  \; \text{(mod $p$)}$ holds, where $x$ is the label on the over-crossing and $y$ and $z$ the other two labels, and
\item[(ii)] at least two labels are distinct.   
\end{itemize}
\end{definition}

If a diagram has a mod $p$ labeling, then multiplying each label by a number $m$ gives a mod $pm$ labeling, so we assume $p$ is always a prime number. 
 
\begin{remark}\label{mod2-coloring}
If $p=2$, then the equation $2x-y-z=0  \; \text{(mod $p$)}$ indicates at each crossing the two under-crossings have the same label, hence all the labels are equal. Therefore, a mod $2$ label for a knot diagram does not exist.
\end{remark}

Given a knot diagram, label each long arc with a variable $x_{i}$. At each crossing we define a relation $2x_{i}-x_{j}-x_{k}=0 \; \text{(mod $p$)}$, if the arc $x_{i}$ crosses over the arcs $x_{j}$ and $x_{k}$. Therefore, a knot can be labeled mod $p$, if this system has a solution mod $p$ such that not all $x_{i}$'s are equal to each other. 

Fix a variable $x_{j}$. Since $x_{i}=1$ for all $i$, is a solution and adding two solutions together forms a new solution, if there was a solution such that not all $x_{i}$'s are equal, then there is a solution with  $x_{j}=0$. Conversely, a nontrivial solution with $x_{j}=0$ results in a labeling of the knot. So we can delete the $j$-th column and look for the nontrivial solutions of the resulting system.

Since we assume the knot diagram is checkerboard colorable, and we know for such a diagram, a linear combination of rows of $B(D)$ is zero, so we can delete the $j$-th row as well. The result is a square matrix, and a nontrivial solution means the determinant should be zero mod $p$. The absolute value of this determinant is $\det(K)$. So, we have the following.

\begin{proposition}\label{coloring}
We can mod $p$ label the knot $K$ if and only if $\det(K)=0 \; \text{(mod $p$)}$. 
\end{proposition}

\begin{corollary} \label{cor:cc}
For a checkerboard colorable knot $K$, $\det(K)$ is an odd integer.
\end{corollary}

\begin{proof}
Combining the \Cref{coloring} and \Cref{mod2-coloring}, the result follows.
\end{proof}

It would be interesting to compare the link determinant defined here with the link determinants defined for checkerboard colorable virtual links in terms of Goeritz matrices \cite{Im-Lee-Lee-2010}.


\setcounter{theorem}{0} 
\setcounter{section}{4} \noindent
\subsection{The Matrix-Tree theorem and an application} \label{sec-4}
In this section, we recall the matrix-tree theorem from \cite{BM} (cf. \cite[Theorem 13.22]{BZH}). Using it, we adapt Crowell's proof \cite{Crowell} to show that the Alexander polynomial of any almost classical alternating link is alternating.

Here, we say a Laurent polynomial is \textit{alternating} if its coefficients alternate in sign. Specifically, a polynomial $f(t)= \sum c_i t^i \in \ZZ[t,t^{-1}]$ is alternating if its coefficients satisfy $(-1)^{i+j} c_i c_j \geq 0.$

The spectacular results concerning the Jones polynomial of classical alternating links are generally not true in the virtual case. For instance, the span of the Jones polynomial is not equal to the crossing number. For example, the knot $K=6.90101$ is alternating and has Jones polynomial $V_{K}(t)=1$.

In \cite{Thistlethwaite-87},
Thistlethwaite proved that the Jones polynomial $V_{L}(t)$ of any
non-split, alternating classical link $L$ is alternating. This result does not extend to virtual links. For example, the virtual knot $K=5.2426$ in \Cref{fig-5-2426} is alternating and has Jones polynomial $V_K(t) = 1/t^2 + 1/t^3 - 1/t^5.$ Since $V_K(t)$ is not alternating, Thistlethwaite's result is not true for virtual links. 

\begin{figure}[h!]
\centering\includegraphics[height=35mm]{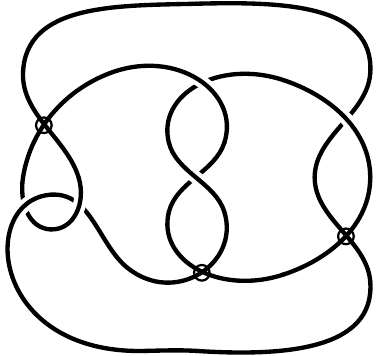}
\caption{\small A virtual knot diagram for $5.2426$.}\label{fig-5-2426}
\end{figure}

Let $L$ be a virtual link. We define the link group $G_{L}$ as in \ref{sec-2}. We use Fox derivatives to define the Jacobian matrix $A$. For virtual knots, the first elementary ideal $\cE_1$ is not necessarily principal. We define the Alexander polynomial $\Delta_{K}(t)$ to be the generator of the smallest principal ideal containing $\cE_{1}$. Since $\ZZ[t,t^{-1}]$ is a gcd domain, it is given by taking the gcd of all the $(n-1)\times(n-1)$ minors of $A$ . If we remove the $i$-th row and $j$-th column of $A$ we denote the corresponding minor by $A_{ij}.$

In \cite{NNST} and \cite{Boden-Nicas-White}, the authors showed for almost classical links, $\cE_{1}$ is principal, and the Alexander polynomial $\Delta_{L}(t)$ is given by taking the determinant of the $(n-1)\times(n-1)$ matrix obtained by removing any row and any column from $A$.

\begin{proposition}
For an almost classical link $L$, the determinant $\det(L)$ is equal to  $|\Delta_{L}(-1)|$.  
\end{proposition}

\begin{proof}
If $D$ is a diagram for $L$, 
the coloring matrix $B(D)$ is exactly the matrix obtained from the Fox Jacobian matrix by replacing $t$ with $-1$. \footnote{These matrices are equal up to multiplication by $\pm 1$ in the rows.}
\end{proof}

\begin{remark}
Since any almost classical knot $K$ is checkerboard colorable, \Cref{cor:cc} shows that $\Delta_{K}(-1)$ is an odd number (see \cite{Boden-Gaudreau-Harper-2017}).
\end{remark}

Next, we state the Matrix-Tree theorem, as proved by Bott and Mayberry in \cite{BM}. Tutte had given an earlier proof in \cite{Tutte-1948}. The result goes back to even earlier work of Kirchhoff, to whom this theorem is usually attributed.\footnote{See \cite{Crowell-na} for references to the early papers on the Matrix-Tree theorem.} 

Let $\Gamma$ be a finite oriented graph with vertices $\{c_i \mid 1\leq i\leq  n\}$ and oriented edges $\{u_{ij}^{\delta}\}$, such that $c_{i}$ is the initial point and $c_j$ the terminal point of $u_{ij}^{\delta}$. Notice that $\delta$ enumerates the different edges from $c_i$ to $c_j$. By a \textit{rooted tree} (with root $c_i$) we mean a subgraph of $n-1$ edges such that every point $c_k$ is terminal point of a path with initial point $c_i$. Let $a_{ij}$ denote the number of edges with initial point $c_i$ and terminal point $c_j$.

\begin{theorem}[Matrix-Tree Theorem] \label{rooted} 
Let $\Gamma$ be a finite oriented graph without loops ($a_{ii}=0$). The principal minor $H_{ii}$ of the graph matrix
$$H(\Gamma)=\left[\begin{matrix}
(\sum_{k\neq 1}a_{k1})&-a_{12}&-a_{13}&\cdots&-a_{1n}\\
-a_{21}&(\sum_{k\neq 2}a_{k2})&-a_{23}&\cdots&-a_{2n}\\
\vdots&\vdots&\vdots&&\vdots\\
-a_{n1}&-a_{n2}&-a_{n3}&\cdots&(\sum_{k\neq n}a_{kn})
\end{matrix}\right],$$
is equal to the number of rooted trees with root $c_{i}$.
\end{theorem}

\begin{corollary}\label{bottm}
Let $\Gamma$ be a finite oriented loopless graph with a \textit{valuation} $f \co \{u_{ij}^{\delta}\}\rightarrow\{-1,1\}$ on edges. Then the principal minor $H_{ii}$ of the matrix $H=[b_{ij}]$, where $$b_{ij}=\begin{cases}\sum_{\delta}f(u_{ij}^{\delta}),\ \ \ i\neq j,\\
-\sum_{k\neq i}b_{ki}, \ \ i=j,\end{cases}$$
satisfies the following equation:
$$H_{ii}=\sum f(\Tr(i)),$$
where the sum is to be taken over all $c_{i}$-rooted trees $\Tr(i)$, and where
$$f(\Tr(i))=\prod_{u_{kj}^{\delta}\in \Tr(i)}f(u_{kj}^{\delta}).$$
\end{corollary}

For a virtual link diagram, there are (at least) two ways one can associate a $4$-valent graph. One way is to consider the diagram $D$ itself. It has vertices for the classical and virtual crossings and edges running from  one classical or virtual crossing to the next. This graph is planar. The other way to associate a graph is to consider vertices only for classical crossings. The key difference is that in general, this graph is not planar. For an alternating diagram $D$, we describe this graph and an orientation on it as follows:

Let $D$ have classical crossings $c_{1},\ldots,c_{n}$. The vertices of $\Gamma$ are $c_{1},\ldots,c_{n}$. At each vertex consider two out-going edges corresponding to the over-crossing arc, and two in-coming edges for the under-crossing arcs (see \Cref{source-sink}). This is called the \textit{source-sink orientation} or the \textit{alternate orientation}. This orientation is possible because  $D$ is alternating, and an out-going edge at the vertex $c_{i}$, should be an in-coming edge for the adjacent vertex.

\begin{remark}\label{rem-source-sink}
In general, any checkerboard colorable diagram $D$ admits a source-sink orientation. In fact, a diagram is checkerboard colorable if and only if it admits a source-sink orientation (see \cite[Proposition 6]{Kamada-skein}). 
\end{remark}

 \begin{figure}[h!]
\centering\includegraphics[scale=0.95]{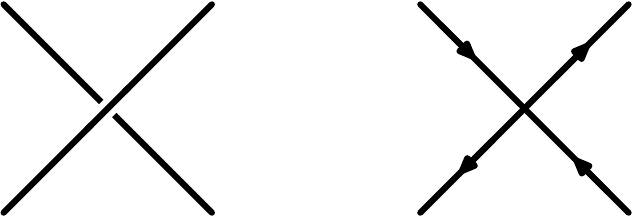}
\caption{\small The source-sink orientation.}\label{source-sink}
\end{figure}

\begin{theorem}\label{almost-alt}
If $L$ is a non-split, almost classical alternating link, then its Alexander polynomial $\Delta_{L}(t)$ is alternating.
\end{theorem}

\begin{proof}
For the unknot the result is obvious.
Assume $D$ has $n\geq 1$ classical crossings. Orient $D$ and enumerate the crossings by $c_1,\ldots,c_n$. Label the long arcs by $g_{1},\ldots,g_{n}$. 
At the crossing $c_i$, label the over-crossing arc $g_{\nu(i)}$ 
and the under-crossing arcs $g_{\lambda(i)}$ and $g_{\rho(i)}$ as in \Cref{crossing-two}.
Define the relation $r_{i}=g_{\lambda(i)}g_{\nu(i)}^{-1} g_{\rho(i)}^{-1}g_{\nu(i)}$. 

\begin{figure}[h!]
\centering\includegraphics[scale=0.90]{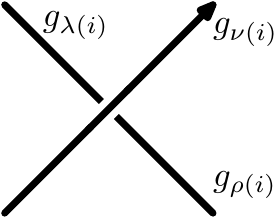}
\caption{\small Arc labels at the crossing $c_i$.} \label{crossing-two}
\end{figure}

Now consider the graph $\Gamma$ associated with $D$, with the source-sink orientation on it. Label the edges by $u_{ij}^{\delta}$. Define the valuation $f$ as follows. At the crossing $c_{j}$, if $u_{ij}^{\delta}$ corresponds to $g_{\lambda(j)}$, then $f(u_{ij}^{\delta})=1$, and if it corresponds to $g_{\rho(j)}$, then $f(u_{ij}^{\delta})=-t$. 

Define the matrix $H$ as in the \Cref{bottm}. Notice that $D$ is alternating and there is a one-to-one correspondence between the classical crossings of $D$ and the set of over-crossing arcs. Therefore, we can choose to label over-crossing arcs, such that $\nu(i)=i$. The matrix $H$ is the transpose of the Jacobian matrix $A$. The Alexander polynomial $\Delta_{L}(t)=A_{ii}=H_{ii}$. By \Cref{bottm}, 
$$H_{ii}=\sum \prod_{u_{kj}^{\delta}\in \Tr(i)}f(u_{kj}^{\delta}).$$
Since $f(u_{kj}^{\delta})=1,\text{or}\ -t$, the product $\prod_{u_{kj}^{\delta}\in \Tr(i)}f(u_{kj}^{\delta})$ is of the form $(-1)^lt^l$ and $H_{ii}$ is an alternating polynomial. Therefore, $\Delta_{L}(t)$ is alternating.   
\end{proof}

\begin{example}
Up to $6$ classical crossings, the almost classical knots in \Cref{table1} do not have alternating Alexander polynomials. Therefore, by \Cref{almost-alt} they do not admit alternating virtual knot diagrams. 
\hfill $\Diamond$ 
\end{example}

\renewcommand{\arraystretch}{1.10}
\begin{table}[h!]
\begin{center}
{\rowcolors{1}{white!80!white!50}{lightgray!70!lightgray!40}
\begin{tabular}{|c|l|} \hline
 $K$ & $\Delta_{K}(t)$  \\ \hline \hline 
    5.2331  &  $t^{2}-1+t^{-1}$      \\ \hline
    6.85091 &  $1+t^{-1}-t^{-2}$     \\ \hline
    6.85774 &  $t-1+t^{-2}$          \\ \hline
    6.87548 &  $-t^{2}+2t+1-t^{-1}$   \\ \hline
    6.87875 &  $t+1-2t^{-1}+t^{-2}$    \\ \hline
    6.89156 &  $2t-1-t^{-1}+t^{-2}$    \\ \hline
    6.89812 &  $t^{2}-2+2t^{-1}$       \\ \hline
    6.90099 &  $t-t^{-1}+t^{-3}$     \\ \hline                    
\end{tabular}
}
\renewcommand{\arraystretch}{1} 
\end{center}
\caption{\small Almost classical knots with non-alternating Alexander polynomials.}\label{table1}
\end{table}

\begin{remark}
Any integral polynomial $\Delta(t)$ of degree $2n$ satisfying $\Delta(1)=1$ and $\Delta(t)= t^{2n} \Delta(t^{-1})$ is the Alexander polynomial for some classical knot, see \cite{Seifert-1935}. In \cite{Fox-1962-b}, Fox asked for a characterization of Alexander polynomials of alternating knots.  If $K$ is an alternating knot, then $\Delta_K(t) = \sum_{j=0}^{2n}(-1)^j a_j t^j$. Fox conjectured that the  Alexander polynomial of any alternating knot satisfies the trapezoidal inequalities: 
$$a_0 < a_1 < \cdots < a_{n-k} = \cdots = a_{n+k} > \cdots > a_{2n}.$$ 
Fox verified this conjecture for alternating knots up to 11 crossings, and it has been verified in many other cases \cite{Jong-2009, Hirasawa-Murasugi, Alrefai-Chbili, Chen-2021}. Despite this progress, Fox's trapezoidal conjecture remains an intriguing open problem. 

Any integral polynomial $\Delta(t)$ satisfying $\Delta(1)=1$ is the Alexander polynomial for some almost classical knot $K$, see \cite{Boden-Chrisman-Gaudreau-2020}. Is there a way to characterize the Alexander polynomials of alternating almost classical knots? Do they satisfy Fox's trapezoidal inequalities? This has been verified for almost classical knots up to six classical crossings. 
\end{remark}

\setcounter{theorem}{0} 
\setcounter{section}{5} \noindent
\subsection{Split alternating virtual links are visibly split}\label{sec-5}

A classical result of Bankwitz \cite{Bankwitz} implies that $\det(L)$ is nontrivial for non-split alternating links. We extend this result to virtual alternating links and apply it to show that an alternating virtual link $L$ is split if and only if it is visibly split.

The weak form of the first Tait Conjecture, namely that every knot having a reduced alternating diagram with at least one crossing is nontrivial, was first proved by Bankwitz \cite{Bankwitz} in 1930; and since then, Menasco and Thistlethwaite \cite{Menasco-This} and Andersson \cite{Andersson} published simpler proofs. Here we outline the proof by Balister et al. \cite{det} and generalize it to alternating virtual links. This result was first proved  for alternating virtual knots  by Cheng \cite[Proposition 3.3]{Cheng-15}.

Consider the graph $\Gamma$ with vertices $\{c_1,\ldots,c_n\}$ as before.

\begin{definition}
The \textit{outdegree} of the vertex $c_i$, denoted $d^{+}(c_i)$, is the number of edges of $\Gamma$ with initial point $c_i$. The \textit{indegree} of the vertex $c_i$, denoted $d^{-}(c_i)$, is the number of edges of $\Gamma$ with terminal point $c_i$. Therefore,
$$d^{+}(c_i)=\sum_{j=1}^n a_{ij}\ \ ,\ \ d^{-}(c_i)=\sum_{j=1}^n a_{ji}.$$ 
\end{definition}

\begin{definition}
A \textit{walk} in a graph is an alternating sequence of vertices and edges, starting with a vertex $c_i$ and ending with a vertex $c_j$. A walk is called a \textit{trail} if all the edges in that walk are distinct. A \textit{circuit} is a trail which starts and ends at a vertex $c_i$. An \textit{Eulerian circuit} is a circuit which contains all the edges of $\Gamma$. A graph $\Gamma$ is called \textit{Eulerian} if it has an Eulerian circuit.     
\end{definition}
 
An Eulerian graph is necessarily connected and has  $d^{+}(c_i)=d^{-}(c_i)$ for every vertex. Let $t_{i}(\Gamma)$ be the number of rooted trees with root $c_i$, then the BEST Theorem is as follows (see \cite{BEST} and \cite[Theorem 13]{Modern}).

\begin{theorem}\label{best}
Let $s(\Gamma)$ be the number of Eulerian circuits of $\Gamma$, then
$$s(\Gamma)=t_{i}(\Gamma)\prod_{j=1}^n(d^{+}(c_j)-1)!$$
\end{theorem}
In particular, if $\Gamma$ is a two-in two-out oriented graph, i.e., $d^{+}(c_i)=d^{-}(c_i)=2$ for every $i$, then by Theorems \ref{rooted} and \ref{best}, 
$$s(\Gamma)= t_i(\Gamma)=H_{ii},\ \ \text{for every}\ i.$$

A vertex $c$ of a graph $\Gamma$ is an \textit{articulation vertex} if $\Gamma$ is the union of two nontrivial graphs with only the vertex $c$ in common. In \cite{det} Balister et al. proved the following result: 

\begin{figure}[h!]
\centering\includegraphics[scale=0.90]{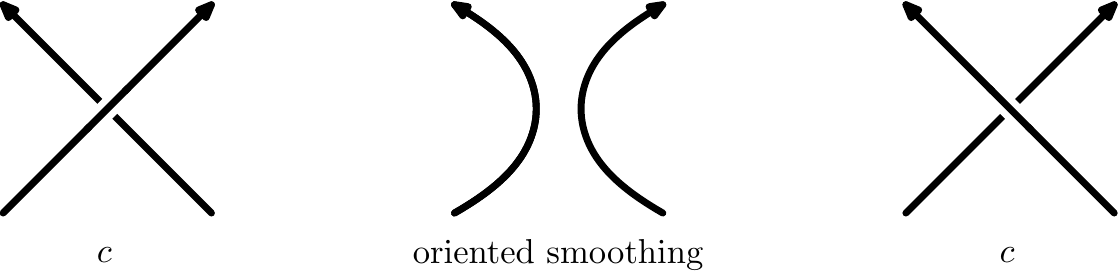}
\caption{\small The oriented smoothing at $c$.} \label{crossing-smoothing}
\end{figure}

\begin{theorem}\label{eulerian}
Let $\Gamma$ be a connected two-in two-out oriented graph with $n\geq 2$ vertices and with no articulation vertex. Then $s(\Gamma)\geq n$.
\end{theorem}
 
Given an oriented virtual  link diagram $D$, recall that the \textit{oriented smoothing} at a crossing $c$ is the diagram with the crossing $c$ removed, see \Cref{crossing-smoothing}. Recall also that a \textit{self-crossing} of $D$ is a crossing where one of the components of the link crosses over itself.

\begin{definition}\label{Dye-reduced}
Let $D$ be an oriented non-split virtual link diagram. Then a self-crossing $c$ is said to be  \textit{nugatory} if the oriented smoothing of $D$ at $c$ is a split link diagram. 

The diagram $D$ is said to be \textit{reduced} if it does not contain any nugatory crossings.
\end{definition}

There is an equivalent definition of nugatory crossing for links in surfaces. Let $c$ be a crossing of a connected link diagram $\cD$ on a surface $\Si$. Then $c$ is said to be \textit{nugatory} if there is a simple closed curve on $\Si$ that separates $\Si$ and intersects $\cD$ exactly once at $c$. 

For classical links, nugatory crossings are always removable. For virtual links, this is no longer true. 
Indeed, there are examples of virtual knots that contain \textit{essential} nugatory crossings, see \cite[Example 20]{Boden-Karimi-Sikora}. For welded links, nugatory crossings are once again always removable, see \Cref{rem-nug} below. 

Recall that associated with an alternating virtual link diagram $D$, there is an oriented two-in two-out graph $\Gamma$. If $D$ has no nugatory crossings, then $\Gamma$ has  no articulation vertex. 

\begin{corollary}
Let $K$ be an almost classical knot and $D$ a reduced alternating diagram for $K$. If $D$ has $n$ classical crossings, then
$$ |\Delta_K(-1)| \geq n.$$
\end{corollary}

\begin{proof}
By the proof of \Cref{almost-alt}, $|\Delta_{K}(-1)|$ counts $t_i(\Gamma)$ the number of rooted trees with root $c_i$ in the oriented graph $\Gamma$, associated with the knot diagram $D$. By \Cref{best}, $t_{i}(\Gamma)=s(\Gamma)$, and the result follows from \Cref{eulerian}.
\end{proof}

\begin{theorem}\label{link determinant}
Let $L$ be a non-split virtual link and $D$ a reduced alternating diagram for $L$. If $D$ has $n$ classical crossings, then the determinant of $L$ satisfies $$\det(L)\geq n.$$
\end{theorem}   

\begin{proof}
Since $D$ is alternating, we can repeat the proof of \Cref{almost-alt}. By \Cref{bottm}, the determinant of $L$ counts the number of spanning trees which is equal to $s(\Gamma)$. The result follows from \Cref{eulerian}. 
\end{proof}

\begin{corollary} \label{cor:split}
Suppose $L$ is a virtual link which admits an alternating diagram $D$ without nugatory crossings. Then $L$ is a split link if and only if  $D$ is a split diagram.
\end{corollary}

\begin{proof}
Clearly, if $D$ is a split diagram, then $L$ is split. Suppose then that $D$ is a non-split  alternating diagram with $n=n(D) >0$ classical crossings. (If $n=0$, then $D$ has one component and is an unknot diagram.) \Cref{link determinant} implies that $\det(L)\geq n$. Hence $\det(L)\neq 0$, and \Cref{split det} shows that $L$ is not split.
\end{proof}
\setcounter{theorem}{0} 
\setcounter{section}{6} \noindent
\subsection{Weakly alternating virtual links}\label{sec-6}
In this section, we extend the results from the previous section to semi-alternating virtual links, defined below. We also give a formula for the link determinant of a connected sum, and we use it to show that a semi-alternating virtual link is split if and only if it is visibly split. 

We begin by reviewing the connected sum of virtual links. 
Suppose $D_1$ and $D_2$ are virtual link diagrams and $p_1, p_2$ are points on $D_1,D_2$ respectively, distinct from the crossings. The connected sum is denoted $D_1 \# D_2$ and is formed by removing small arcs from $D_1$ and $D_2$ near the basepoints and joining them with trivial unknotted arcs. If $D_1$ and $D_2$ are oriented, then the arcs are required to preserve orientations. 
The connected sum depends on the choice of diagrams and basepoints. It does not lead to a well-defined operation on virtual links.

\begin{definition}\label{semi-alternating}
A virtual link diagram $D$ is said to be \textit{semi-alternating} if it can be written 
$D=D_1 \# \cdots \# D_n$, a connected sum of alternating virtual link diagrams $D_1,\ldots, D_n$. 
\end{definition}

The set of semi-alternating virtual links includes, as a proper subset, those that can be represented as weakly alternating links in thickened surfaces, see \cite[\S 5]{Boden-Karimi-Sikora}.

Every semi-alternating virtual link diagram is checkerboard colorable. This follows from the fact that every alternating virtual link diagram is checkerboard colorable (see \cite[Lemma 7]{Kamada-2002}), and the observation that the connected sum of two or more checkerboard colorable diagrams is checkerboard colorable. 

\begin{definition}
A virtual link is said to be \textit{$w$-split} if it is welded equivalent to a split virtual link. 
\end{definition}
Clearly, any virtual link that is split is necessarily $w$-split, but there are virtual links that are $w$-split but not split. For example, consider the virtual link $L$ whose Gauss diagram appears on the left of \Cref{splitlink}. Using forbidden moves, it is seen to be welded equivalent to the split classical link $8_{20} \cup \bigcirc$ shown on the right. Thus $L$ is $w$-split.

Let $L'=8_{20} \cup \bigcirc$ be the split classical link shown on the right of \Cref{splitlink}. Using \cite[Definition 3.1 \& Theorem 3.5]{Lick}, one can see that its Jones polynomial satisfies
\begin{equation*}
\begin{split}
(t^{-1/2}-t^{1/2}) V(L') &=  (t-t^{-1}) V(8_{20}) \\ 
& = t^{-6} -t^{-5} -t^{-3}+2t-t^2.
\end{split}
\end{equation*}
In particular,  $(t^{-1/2}-t^{1/2}) V(L')$ lies in $\ZZ[t,t^{-1}].$
(This is true for any classical link with two components.)
On the other hand,  direct computation shows that the Jones polynomial of  $L$ satisfies
\begin{equation*}
\begin{split}
(t^{-1/2} -t^{1/2}) V(L)  = & -t^{-1/2} + 3t^{-3/2} + 2t^{-2} - 3t^{-5/2} - 3t^{-3} + 2t^{-7/2} +\\
& \quad 3t^{-4} - t^{-9/2} - 3t^{-5} - t^{-11/2} + t^{-6} + t^{-13/2}.
\end{split}
\end{equation*}
Since $(t^{-1/2} -t^{1/2}) V(L)$ does not lie in $\ZZ[t,t^{-1}]$, $L$ cannot be virtually equivalent to a classical link. On the other hand, if $L$ were split, then it would be virtually equivalent to $L'$. Since that is not the case, we see that $L$ is non-split. 

\begin{figure}[h!]
\centering\includegraphics[scale=1.4]{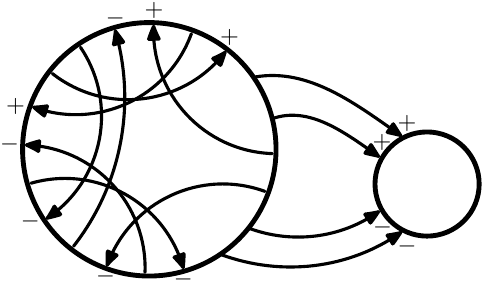}\hspace{1.5cm}
\centering\includegraphics[scale=0.70]{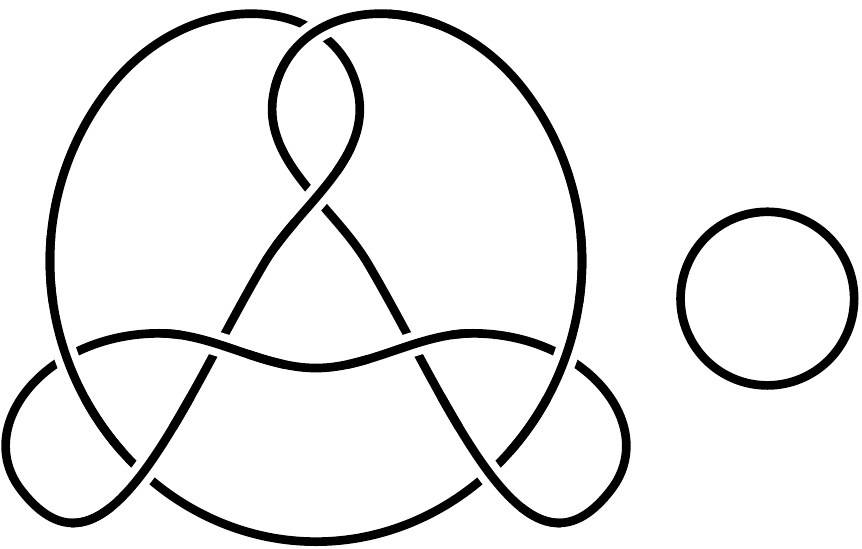}
\caption{\small The Gauss diagram for a non-split virtual link $L$ (left). Notice that $L$ is $w$-split; in fact it is welded equivalent to the link $8_{20} \cup \bigcirc$ (right).}\label{splitlink}
\end{figure}

\begin{proposition} \label{w-split}
If $L$ is $w$-split, then $\det(L) = 0.$
\end{proposition}
\begin{proof}
This follows directly from \Cref{split det} and the fact that $\det(L)$ is an invariant of welded links.
\end{proof}

There is a nice geometric interpretation of $\det(L)$ in terms of two-fold branched covers. Let $T =\Tube(L)$ be the ribbon torus link associated to $L$, and let $X$ be the two-fold cover of $S^4$ branched along $T$. Using the isomorphism $\pi_1(S^4 \sm T) \cong G_L$, one can identify $\pi_1(X)$ with the quotient of $G_L$ under the relation $x_i^2 =1$ for each generator in \eqref{eq:Wirt}. The coloring matrix is then a presentation matrix for $H_1(X)$. Therefore, $\det(L) = |H_1(X)|$ if it is  finite, and $\det(L) = 0$ if $H_1(X)$ is infinite. Here homology groups are taken with $\ZZ$ coefficients. Note that, if $L$ is split, then $H_1(X)$ is infinite. This gives an alternative explanation for Propositions \ref{split det} and \ref{w-split}.

\begin{theorem} \label{product}
If $D=D_1 \# D_2$ is a connected sum of two checkerboard colorable virtual link diagrams, then  $\det(D) =\det(D_1) \det(D_2).$
\end{theorem}
\begin{proof}
If $D_1$ or $D_2$ is split, then $D$ is split and $\det(D)=0= \det(D_1)\det(D_2).$ Therefore, we can assume that $D_1$ and $D_2$ are non-split. 

There is a proof which is direct and elementary but long. We present an alternative proof that is shorter and makes use of the interpretation of $\det(D)$ as the order of the first homology of the two-fold cover of $S^4$ branched along $\Tube(D)$. In the following, all homology groups are taken with $\ZZ$ coefficients.

Let $D_1$ and $D_2$ be checkerboard colorable virtual link diagrams, and let $X_1, X_2,$ and $X$ be the two-fold covers of $S^4$ branched along $\Tube(D_1), \Tube(D_2)$ and $\Tube(D)$, respectively.
We can then write $X_1=A_1\cup B_1, X_2=A_2\cup B_2$, and $X=A_1\cup A_2$. Here 
$A_i$ is the double cover of $D^4$ branched along the knotted annulus which is part of $\Tube(D_i)$ for $i=1,2$, and $B_i$ is the double cover of $D^4$ branched along trivial annulus. In particular, $A_i = X_i \sm \Int(B_i)$ for $i=1,2$.  By \cite[Corollary 4.3]{Kirby-Akbulut}, $B_i$ is diffeomorphic to $S^2\times D^2$, and  $H_1(B_i)=0$ for $i=1,2.$ 

Let $M=A_1\cap B_1=A_2\cap B_2=A_1\cap A_2$. Then $M$ is the 3-manifold obtained as the double cover of $S^3$ branched along the two component unlink. Thus $M$ can be identified with the boundary of $B_1$ (or $B_2$) and is diffeomorphic to $S^2 \times S^1$. Thus $H_1(M) \cong \ZZ.$ 
 
Now consider the decompositions $X_1=A_1 \cup B_1$, $X_2=A_2 \cup B_2$, and $X =A_1 \cup A_2$, along with their Mayer-Vietoris sequences in reduced homology:
\begin{equation}\label{eq-MV}
\begin{array}{c}
  \xymatrix{\cdots \ar[r]^{} & H_1(A_1\cap B_1)\ar[r]^{\varphi_1\quad} & H_1(A_1)\oplus H_1(B_1)\ar[r]^{\qquad \psi_1} & H_1(X_1)\ar[r]^{}&0.}\\
   \xymatrix{\cdots \ar[r]^{} & H_1(A_2\cap B_2)\ar[r]^{\varphi_2\quad} & H_1(A_2)\oplus H_1(B_2)\ar[r]^{\qquad \psi_2} & H_1(X_2)\ar[r]^{}&0.} \\
 \xymatrix{\cdots \ar[r]^{} & H_1(A_1\cap A_2)\ar[r]^{\varphi\quad} & H_1(A_1)\oplus H_1(A_2)\ar[r]^{\qquad \psi} & H_1(X)\ar[r]^{}&0.}
\end{array} 
\end{equation}

Note that $H_1(A_1\cap B_1) \cong \ZZ \cong H_1(A_2\cap B_2) = H_1(A_1\cap A_2)$.

We claim that the maps $\varphi_1,\varphi_2$ and $\varphi$ are zero. We prove this for $\varphi$;  the argument for the other cases is similar. It suffices to show that the maps $H_1(A_1\cap A_2) \to H_1(A_i)$  induced by inclusion are zero for $i=1,2$. 

Take two points in $S^3$, one on each component of the unlink, and join them by an arc in $S^3$ that is otherwise disjoint from the link. The arc lifts to a loop in the double branched cover, and the loop is a generator of $H_1(A_1 \cap A_2)$. However, when pushed into $A_1$,  the loop does not link the annulus in $D^4$, so it is trivial in $H_1(A_1)$. A similar argument shows it is also trivial in $H_1(A_2)$. Therefore, the maps $H_1(A_1\cap A_2) \to H_1(A_j)$ are zero for $j=1,2$, and it follows that $\varphi=0.$

From the claim, it follows that $\psi_1,\psi_2,$ and $\psi$ are isomorphisms. Using \eqref{eq-MV} and the fact that $H_1(B_1)=0=H_1(B_2)$, we deduce that 
$$H_1(X_1) \cong H_1(A_1), \; H_1(X_2) \cong H_1(A_2), \text{ and } H_1(X)\cong H_1(A_1) \oplus H_1(A_2).$$ 
Therefore,
\begin{equation*}
\begin{split}
\det(D)&=|H_1(X)|,\\
&=|H_1(A_1)| \cdot |H_1(A_2)|,\\
&=|H_1(X_1)|\cdot |H_1(X_2)|,\\
&=\det(D_1)\det(D_2),
\end{split}
\end{equation*}  
and this completes the proof.  
\end{proof}

\begin{remark} \label{rem-nug}
We claim that, for welded links, nugatory crossings are always removable. Let $D$ be a diagram with a nugatory crossing $c$ as in \Cref{nugatory} (left). Using forbidden moves, we can transform $D$ by pulling the over-crossing arc off $c$, as in \Cref{nugatory} (middle). Thus, $D$ is welded equivalent to the diagram with $c$ removed. Alternatively, one can remove $c$ by making it virtual, as in \Cref{nugatory} (right). At the level of Gauss diagrams, this is equivalent to deleting the chord associated to $c$. 
\end{remark}

\begin{figure}[h!]
\centering\includegraphics[scale=1.40]{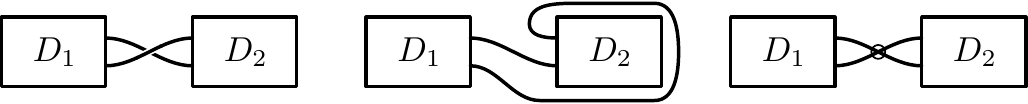}
\caption{\small For welded links, nugatory crossing are removable.} \label{nugatory}
\end{figure}

\begin{proposition} \label{lem-split}
Let $D$ be a virtual link diagram. If $D$ is non-split and alternating, then $\det(D) \neq 0$. 
\end{proposition}

\begin{proof}
In general, if $D$ is a non-reduced virtual link diagram, then successively removing all the nugatory crossings will produce a reduced diagram $D'$ welded equivalent to $D$. If $D$ is non-split, then $D'$ will be too. If $D$ is alternating, then $D'$ will be semi-alternating.

Assume to the contrary that $\det(D)=0.$  Then $\det(D')=0$ since $\det(\cdot)$ is an invariant of welded type. Since $D$ is non-split and alternating, it follows that
$D'$ is reduced, non-split, and semi-alternating.  
Therefore, we can write $D'=D'_1 \# \cdots \# D'_n$, where $D'_1,\ldots, D'_n$ are all reduced alternating diagrams. By \Cref{product}, 
$$0=\det(D')= \det(D'_1)\cdots \det(D'_n),$$ 
thus $\det(D'_i)=0$ for some $1\leq i \leq n.$ Since $D'_i$ is reduced and alternating,  $\det(D'_i)=0$ implies that $D'_i$ is split.  It follows that $D'$ is split, which implies that $D$ is split, giving the desired contradiction.
\end{proof}

\begin{corollary} 
Suppose $L$ is a virtual link which admits a semi-alternating diagram $D$, possibly with nugatory crossings. Then $L$ is $w$-split if and only if $D$ is a split diagram.
\end{corollary}

\begin{proof}
Clearly if $D$ is split, then $L$ is split and also $w$-split. 

On the other hand, suppose $D$ is non-split. Since $D$ is semi-alternating, we can write $D=D_1 \# \cdots \# D_n$, where $D_1,\ldots, D_n$ are all non-split, alternating diagrams.  \Cref{lem-split} implies that $\det(D_i) \neq 0$ for $i=1,\ldots, n.$ \Cref{product} implies that $\det(D)= \prod_{i=1}^n \det(D_i) \neq 0.$
Therefore, $\det(L) \neq 0$, and by \Cref{w-split}, it follows that $L$ is not $w$-split.
\end{proof}
 
\setcounter{theorem}{0} 
\setcounter{section}{7} \noindent
\subsection{The Tait conjectures for welded links}\label{sec-7}

In his early work on knot tabulation, Tait formulated three far-reaching conjectures on reduced alternating classical link diagrams \cite{Tait}. (A link diagram is \textit{reduced} if it does not contain a nugatory crossing.) They assert that, for a non-split link, any two reduced alternating diagrams have the same crossing number, the same writhe, and are related by a sequence of flype moves. The first two statements were famously solved by Kauffman, Murasugi, and Thistlethwaite using the recently discovered Jones polynomial \cite{Kauffman-87, Murasugi-871, Thistlethwaite-87}, and the third statement was subsequently proved by Menasco and Thistlethwaite \cite{Tait3}. The three Tait conjectures lead to a simple and effective algorithm for tabulating alternating knots and links that has been implemented \cite{RF-2004, Knotilus}.

It is an interesting question whether similar results hold for virtual and/or welded links. For example, analogues of the first and second Tait conjectures have been established for virtual links using the Jones-Krushkal polynomial and skein bracket, see \cite{Boden-Karimi-2019, Boden-Karimi-Sikora}. 

\begin{problem}
Is the Tait flype conjecture true for alternating virtual links?
\end{problem}

\begin{figure}[h!]
\centering\includegraphics[scale=1.50]{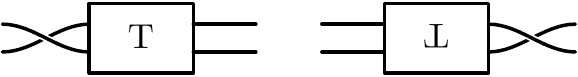}
\caption{\small The flype move.}\label{flype}
\end{figure}

The flype move is depicted in \Cref{flype}. By assumption, the tangle $T$ does not contain any virtual crossings. Allowing the tangle to contain virtual crossings results in a more general move called a \textit{virtual flype move}. The virtual flype move does not, in general, preserve the virtual link type, for example, see \cite{Kamada-17, ZJZ-04}.

It is unknown whether the Tait conjectures hold for welded links. More generally, what conditions must the invariants of welded link satisfy in order for it to be alternating?

Since $\det(L)$ is an invariant of welded links, any checkerboard colorable virtual $L$ with $\det(L)\neq 1$ is nontrivial as a welded link. In particular, \Cref{link determinant} applies to show that any non-split virtual link represented by a reduced, alternating diagram has $\det(L) \neq 1$ and therefore, is nontrivial as a welded  link.

The Alexander polynomial $\De_L(t)$ is also an invariant of the welded type. Therefore, if $L$ is almost classical and $\De_L(t)$ is not alternating, then $L$ is not welded equivalent to an alternating link. 

\begin{figure}[h!]
\centering\includegraphics[height=22mm]{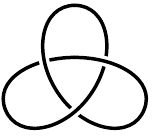} \qquad  
\includegraphics[height=24mm]{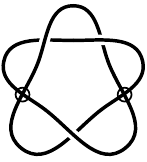}\qquad  
\includegraphics[height=24mm]{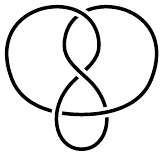}\qquad  
\includegraphics[height=26mm]{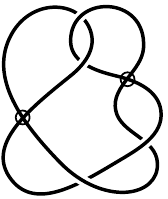}\qquad  
\includegraphics[height=28mm]{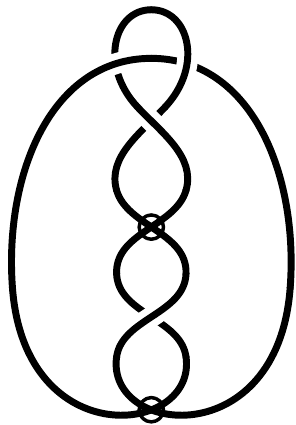}  

\caption{\small Alternating welded knots with 3 and 4 classical crossings.}\label{alternating-welded-knots}
\end{figure}

\Cref{alternating-welded-knots} shows the five alternating welded knots with up to four classical crossings. All the others can be ruled out using the consideration that $\det(K) \geq n$, the crossing number.

\begin{problem}
Is the first Tait conjecture true for alternating welded links? 
\end{problem}

One can find examples of virtual knots which are non-alternating but which become alternating after adding one crossing. For example, consider Examples 19 and 20, \cite{Boden-Karimi-Sikora}. The first is non-alternating and has six crossings; the second is alternating and is obtained by adding a nugatory crossing. The two virtual knots are welded equivalent (see \Cref{nugatory}). We conjecture that there exist welded knots which are alternating, but every minimal crossing diagram for them is non-alternating.

\begin{problem}
Is it possible for an alternating welded knot to represent a non-alternating classical knot?
\end{problem}

\begin{figure}[h!]
\centering\includegraphics[height=30mm]{Figures/4-106.pdf} \qquad \qquad
\includegraphics[height=28mm]{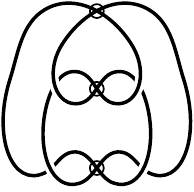}
\caption{\small Reduced alternating diagrams for the virtual knots $4.106$ and $4.107$.}\label{alternating-diagrams}
\end{figure}

Interestingly, there are pairs of reduced alternating virtual knot diagrams which are equivalent as welded knots but distinct as virtual knots. In particular, this implies that Tait's second conjecture is not true for welded knots.

\begin{figure}[h!]
\centering
\includegraphics[height=21mm]{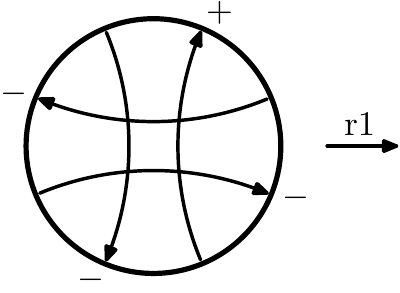}
\includegraphics[height=21mm]{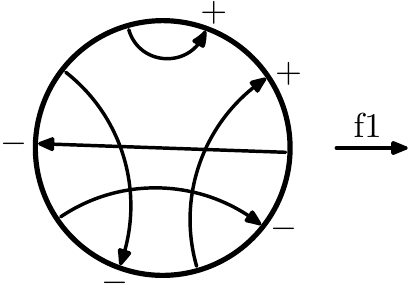}
\includegraphics[height=21mm]{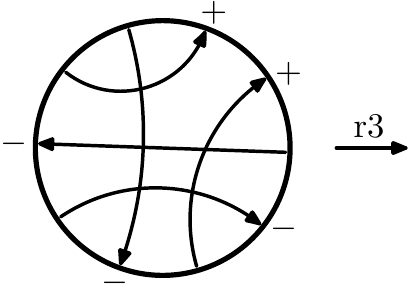}
\includegraphics[height=21mm]{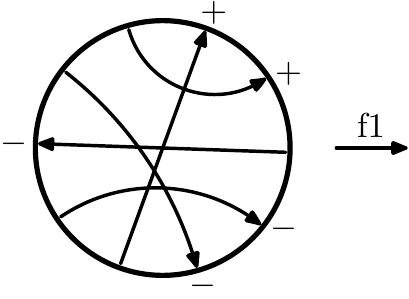}
\includegraphics[height=21mm]{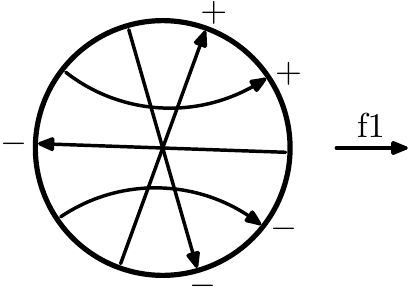}
\includegraphics[height=21mm]{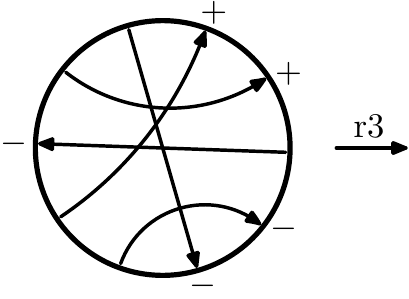}
\includegraphics[height=22mm]{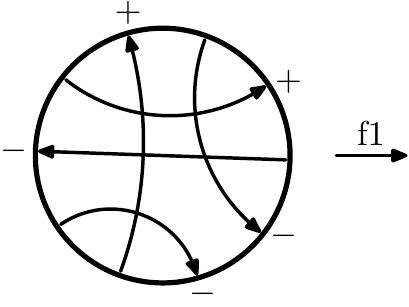}
\includegraphics[height=22mm]{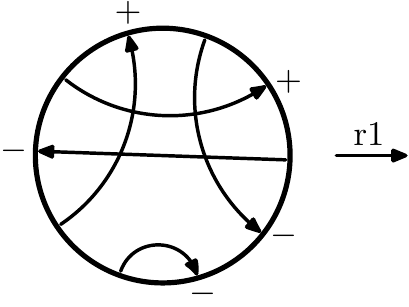}
\includegraphics[height=22mm]{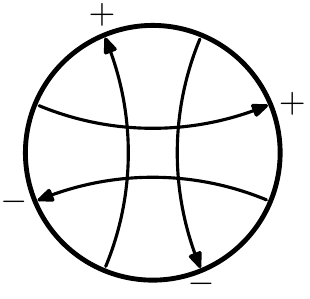}
\caption{\small A sequence of moves on the Gauss diagrams, starting from a diagram of $4.106$ ending in a diagram for $4.107$.}\label{Gauss-moves}
\end{figure}

For example, consider the virtual knots $4.106$ and $4.107$ in \Cref{alternating-diagrams}.
Both are reduced alternating diagrams, but the diagram for $4.106$ has writhe $w=-2$
whereas the diagram for $4.107$ has writhe $w=0.$ Tait's second conjecture holds for 
reduced alternating virtual knot diagrams \cite{Boden-Karimi-2019}, and thus comparing the writhes tells us these two are distinct as virtual knots. However, these diagrams are equivalent as welded knots
(see \Cref{Gauss-moves} and \Cref{diagram-moves}). Since both diagrams are reduced and alternating, this shows that the writhe of a reduced alternating diagram is not invariant under welded equivalence.

This implies that the second Tait conjecture is not true in the welded category. Since the Tait flype move preserves the writhe, this also shows that the Tait's third conjecture, if true, must take a different form in the virtual and welded settings.

\begin{figure}[h]
\centering
\includegraphics[height=28mm]{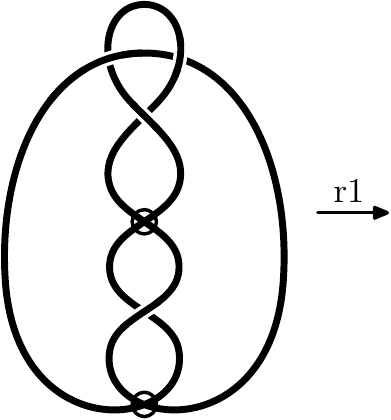}
\includegraphics[height=28mm]{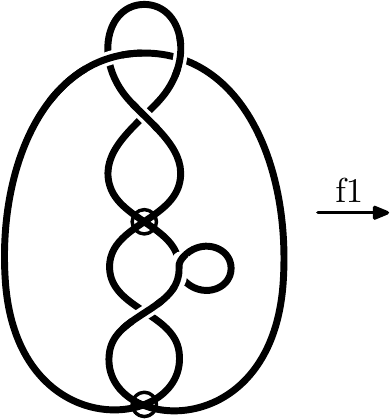}
\includegraphics[height=28mm]{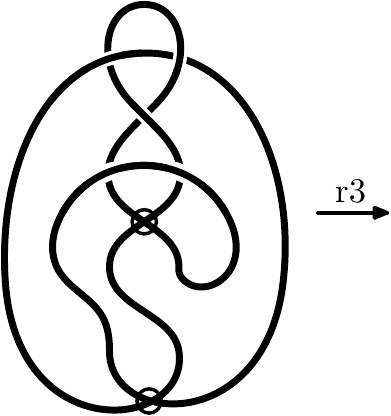}
\includegraphics[height=28mm]{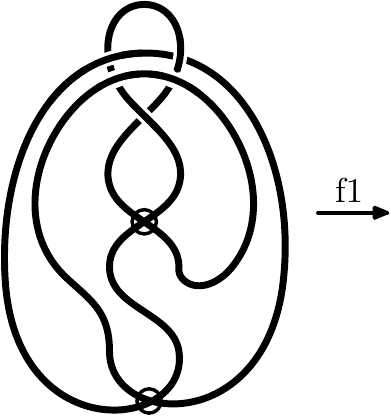}
\includegraphics[height=24mm]{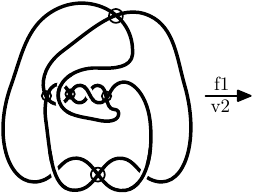}
\includegraphics[height=24mm]{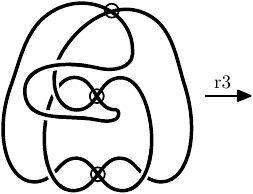}
\includegraphics[height=24mm]{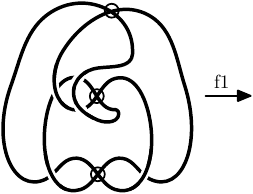}
\includegraphics[height=24mm]{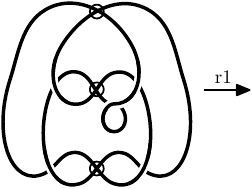}
\includegraphics[height=24mm]{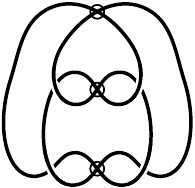}
\caption{\small A sequence of moves on the virtual diagrams, starting from a diagram of $4.106$ ending in a diagram for $4.107$. The fourth and fifth diagrams are related by an $f1$ move, and this is seen by comparing Gauss diagrams.}\label{diagram-moves}
\end{figure}


\subsection*{Acknowledgements}
This paper is based on several ideas in the Ph.D. thesis of the second author \cite{Karimi}.
The authors would like to thank Robin Gaudreau, Andy Nicas, Will Rushworth, and Adam Sikora for their helpful comments and feedback.  They would also like to thank the referee for their input.
The first author gratefully acknowledges grant funding from the Natural Sciences and Engineering Research Council of Canada.
 
\newcommand{\etalchar}[1]{$^{#1}$}

\end{document}